\documentclass[10pt]{article}

\usepackage{tikz}
\usetikzlibrary{matrix,arrows,decorations.pathmorphing,shapes.geometric}

\usepackage{etex}
\usepackage{subcaption}
\usepackage[utf8]{inputenc}
\usepackage{a4wide}
\usepackage{graphicx}
\usepackage{wrapfig}
\usepackage[hmarginratio={1:1},     
vmarginratio={1:1},     
textwidth=17cm,        
textheight=23cm,
heightrounded,]{geometry}

\usepackage{enumitem}

\usepackage{amsmath,accents}
\usepackage{amsthm}
\usepackage{amssymb}

\usepackage{mathrsfs, mathtools}

\usepackage[all,cmtip]{xy}

\let\emph\undefined
\newcommand{\emph}[1]{\textsl{#1}}

\usepackage{hyperref}

\numberwithin{equation}{section}
\usepackage{mathtools}
\mathtoolsset{showonlyrefs}
\numberwithin{equation}{section}

\newtheoremstyle{style1}
{13pt}
{13pt}
{}
{}
{\normalfont\bfseries}
{.}
{.5em}
{}

\theoremstyle{style1}

\newtheorem{definition}{Definition}[section]

\newtheorem{remark}[definition]{Remark}

\newcommand{\catf}[1]{{\mathsf{#1}}}

\usepackage{tocloft}

\newtheoremstyle{style2}
{13pt}
{13pt}
{\slshape}
{}
{\normalfont\bfseries}
{.}
{.5em}
{}

\theoremstyle{style2}

\newtheorem{lemma}[definition]{Lemma}
\newtheorem{theorem}[definition]{Theorem}
\newtheorem{proposition}[definition]{Proposition}
\newtheorem{corollary}[definition]{Corollary}

\usepackage{tikz}
\usetikzlibrary{matrix,arrows,decorations.pathmorphing,shapes.geometric}
\usepackage{tikz-cd}

\usepackage{enumitem}

\usepackage{needspace}
\newcommand{\spaceplease}{\needspace{5\baselineskip}}

\newcommand{\DW}{\operatorname{\normalfont DW}}

\newcommand{\Red}{\catf{Red}}
\newcommand{\cen}{\catf{C}}
\newcommand{\tft}{\catf{TFT}}
\newcommand{\gtft}{G\text{-}\catf{TFT}}
\newcommand{\PBun}{\catf{PBun}}

\newcommand{\Aut}{\operatorname{Aut}}

\newcommand{\id}{\operatorname{id}}

\newcommand{\ev}{\operatorname{ev}}

\newcommand{\fvs}{\catf{FinVect}}
\newcommand{\vs}{\catf{Vect}}
\newcommand{\Tvs}{{2\catf{Vect}}}

\newcommand{\U}{\text{\normalfont U}}

\newcommand{\Cob}{{\catf{Cob}}}

\let\to\undefined
\newcommand{\to}{\longrightarrow}
\let\mapsto\undefined
\newcommand{\mapsto}{\longmapsto}
\newcommand{\TVectBun}{2\catf{VecBunGrpd}}

\newcommand{\Par}{\catf{Par}}

\let\colon\undefined\newcommand{\colon}{:}

\DeclareMathSymbol{\Phiit}{\mathalpha}{letters}{"08} 
\DeclareMathSymbol{\Psiit}{\mathalpha}{letters}{"09}
\DeclareMathSymbol{\Sigmait}{\mathalpha}{letters}{"06}
\DeclareMathSymbol{\Xiit}{\mathalpha}{letters}{"04}
\DeclareMathSymbol{\Piit}{\mathalpha}{letters}{"05}\let\Pi\undefined\newcommand{\Pi}{\Piit}
\DeclareMathSymbol{\Gammait}{\mathalpha}{letters}{"00}
\DeclareMathSymbol{\Omegait}{\mathalpha}{letters}{"0A}\let\Omega\undefined\newcommand{\Omega}{\Omegait}
\DeclareMathSymbol{\Upsilonit}{\mathalpha}{letters}{"07}
\DeclareMathSymbol{\Thetait}{\mathalpha}{letters}{"02}
\DeclareMathSymbol{\Lambdait}{\mathalpha}{letters}{"03}\let\Lambda\undefined\newcommand{\Lambda}{\Lambdait}

\let\Phi\undefined\newcommand{\Phi}{\Phiit}
\let\Sigma\undefined\newcommand{\Sigma}{\Sigmait}
\let\Psi\undefined\newcommand{\Psi}{\Psiit}
\let\Gamma\undefined\newcommand{\Gamma}{\Gammait}

\begin{document}

	\vspace*{-2cm}
	\begin{flushright}
		\small
		{\sf [ZMP-HH/20-9]} \\
		\textsf{Hamburger Beiträge zur Mathematik Nr.~833 }\\
		\textsf{April 2020}
	\end{flushright}
	
	\vspace{5mm}
	
	\begin{center}
		\textbf{\LARGE{Dimensional Reduction,  Extended Topological \\[0.5ex] Field Theories   and Orbifoldization }}\\
		\vspace{1cm}
		{\large Lukas Müller $^{a}$} \ \ and \ \ {\large Lukas Woike $^{b}$}
		
		\vspace{5mm}
		
	{\slshape $^a$ Max-Planck-Institut f\"ur Mathematik\\
Vivatsgasse 7, D\,--\,53111 Bonn}\\
{\tt lmueller4@mpim-bonn.mpg.de\ }
	\\[7pt]
		{\slshape $^b$ Fachbereich Mathematik\\ Universit\"at Hamburg\\
			Bereich Algebra und Zahlentheorie\\
			Bundesstra\ss e 55, \ D\,--\,20\,146\, Hamburg }\\
		{\tt  lukas.jannik.woike@uni-hamburg.de\ }
	\end{center}
	\vspace{0.3cm}
	\begin{abstract}\noindent 
	We prove a  decomposition formula for the dimensional reduction of an extended topological field theory that arises as an orbifold of an equivariant topological field theory. Our decomposition formula can be expressed in terms of a categorification of the integral with respect to groupoid cardinality. The application of our result to topological field theories of Dijkgraaf-Witten type proves a recent conjecture of Qiu-Wang.
	\end{abstract}
	
	\tableofcontents
\section{Introduction}
This short note is concerned with the application of 
recently developed orbifoldization techniques for extended topological field theories to the subject of dimensional reduction.
More concretely, we prove a decomposition formula for the dimensional reduction of a topological orbifold theory.

Topological  field theories 
form a class of (quantum) field  theories that can be mathematically axiomatized as symmetric monoidal functors out of the cobordism category \cite{atiyah}. For their relevance in physics, we refer e.g.\ to \cite{kapustin} and to \cite{rowellwang}  more specifically for applications in quantum computing. 

Additionally, topological field theories have attracted interest in pure mathematics 
thanks to the manifold invariants they provide \cite{rt2,turaevviro}
and the deep connections between three-dimensional topological field theory and representation theory  \cite{turaev1,BDSPV153D} that allow us to use
low-dimensional topology
for the
 study of algebraic objects like modular categories.
Moreover, extended topological field theories and the cobordism hypothesis have become one of the driving forces behind many important developments in the theory of higher categories \cite{lurietft}.

Topological field theories come in different flavors. In particular, there are numerous interesting and well-studied types of topological field theories featuring decorated cobordisms. In the monograph \cite{turaevhqft}, this decoration is chosen to be a map to some topological space, the \emph{target space}. 
When the target space is the classifying space $BG$ of a (finite) group $G$, this leads to \emph{$G$-equivariant topological field theories}. Constructions of classes of examples are available in \cite{htv,maiernikolausschweigerteq,hrt,MuellerWoike1}. Of course, the decoration with maps to $BG$ amounts to a decoration with $G$-bundles. 
In this note, we will be concerned with once-extended (in the sequel, just `extended' for brevity) $G$-equivariant topological field theories
which are formally described as   symmetric monoidal functors $Z:G\text{-}\Cob_n \to \Tvs$ from the symmetric monoidal oriented $G$-cobordism bicategory $G\text{-}\Cob_n$ to the symmetric monoidal bicategory of Kapranov-Voevodsky 2-vector spaces
(these will always be over the complex numbers in this note). As in \cite{turaevhqft}, these functors  additionally have to satisfy a homotopy invariance property.

In the papers \cite{swofk,swpar,extofk}, a topological orbifold  construction has been developed. For a given finite group $G$, this construction assigns to an extended 
 $G$-equivariant topological field theory of any dimension a non-equivariant topological field theory of the same dimension, i.e.\
 it takes the form of a functor\begin{align}
 - / G \ : \ \gtft \to \tft \ , \quad Z \mapsto Z/G \ . 
 \end{align}
 The image of a $G$-equivariant topological field theory $Z$ under this functor is called its \emph{orbifold theory} and denoted by $Z/G$.
 The construction combines the computation of homotopy fixed points with the integration over essentially finite groupoids.
When restricted to the circle in the three-dimensional case, the topological orbifold construction amounts to the categorical orbifoldization \cite{mueger,kirrilovg04,centerofgradedfusioncategories} thereby opening a profitable topological perspective on this important algebraic operation.

The purpose of this note is to give a decomposition formula for the dimensional reduction of such an orbifold theory $Z/G$.
For a fixed closed oriented
manifold $T$ of dimension $r$, the dimensional reduction on $T$ (hereafter also referred to as \emph{$T$-reduction})
assigns to an $n+r$-dimensional extended  topological field theory $E$
an $n$-dimensional extended topological field theory $\Red_T E$ by evaluation of $E$ on manifolds of the form $M\times T$, where $M$ is of dimension $n$, $n-1$ and $n-2$, i.e.\ $\Red_T E = E(-\times T)$.  

In order to formulate our main result,
we extend the notion of dimensional reduction to \emph{equivariant} topological field theories such that 
for a given $n+r$-dimensional $G$-equivariant topological field theory $Z:G\text{-}\Cob_{n+r} \to \Tvs$,  the $T$-reduction can be seen to produce a family
of equivariant topological field theory indexed by all $G$-bundles over $T$. The family member $(\Red_T Z)^P$ belonging to a $G$-bundle $P\to T$ is an $\Aut(P)$-equivariant topological field theory of dimension $n$. Here $\Aut(P)$ is the group of bundle automorphisms of $P$.

Our main result asserts that the $T$-reduction of $Z/G$ is a sum over the $\Aut(P)$-orbifolds of the theories $(\Red_T Z)^P$:
\\[2ex]
\noindent {\textbf{Theorem~\ref{thmdimred2}.} \slshape
		Let $T$ be a closed oriented $r$-dimensional manifold, $G$ a finite group and $Z\colon G\text{-}\Cob_{n+r}\longrightarrow \Tvs$ an $n+r$-dimensional extended $G$-equivariant topological field theory. Then we have a canonical isomorphism
	\begin{align}
	\Red_T \frac{Z}{G} \cong \bigoplus_{ \substack{ \text{isomorphism classes $[P]$} \\ \text{of $G$-bundles over $T$} }  }    \frac{(\Red_T Z)^P}{\Aut(P)}      
	\end{align}
	of $n$-dimensional extended topological field theories.\normalsize} \\[2ex]
From the assumptions, it follows that this is a finite decomposition.

At the heart of the proof lies a compatibility of dimensional reduction and orbifoldization (Theorem~\ref{Thm:main}).
The right hand side can be interpreted as a categorified integral with respect to groupoid cardinality (Remark~\ref{remcatint}).

We apply and illustrate our results for theories of Dijkgraaf-Witten type \cite{dw,freedquinn,morton1} which naturally have an orbifold description \cite{swofk,MuellerWoike1}: Let $\DW_G^\omega$ be the $n+r$-dimensional extended Dijkgraaf-Witten theory for a finite group $G$ and cohomological twist $\omega \in Z^{n+r}(BG;\U(1))$. Then its reduction on a closed oriented $r$-dimensional manifold $T$ is given by
\begin{align}
\Red_T \DW_G^\omega  \cong \bigoplus_{ \substack{ \text{isomorphism classes $[P]$} \\ \text{of $G$-bundles over $T$} }  } \DW_{\Aut(P)}^{\tau \omega|_{   \Aut (P)  }}   \ ,
\end{align}
where $\tau \omega|_{   \Aut (P)  }$ is the restriction of the transgression $\tau \omega \in Z^n(BG^T;\U(1))$ of $\omega$ to the space $BG^T$ of maps $T\to BG$, see \cite{willerton} for the concept of transgression. 
In the case of $T=\mathbb{S}^1$ and trivial $\omega$, we obtain
\begin{align}
\Red_{    \mathbb{S}^1       } \DW_G \cong \bigoplus_{[g] \in G/G   }  \DW_{\cen(g)}  
\end{align}
and thereby express the dimensional reduction of the Dijkgraaf-Witten theory for $G$ as a sum over  Dijkgraaf-Witten theories for the centralizer subgroups $\cen(g)$, where $g$ runs over all conjugacy classes of $G$. 
This decomposition result was recently conjectured by Qiu-Wang in \cite{qw} for a reduction from dimension $(4,3,2)$ to $(3,2,1)$.

		\subparagraph{Acknowledgments.}
	We would like to thank 
	  Christoph Schweigert
	for helpful discussions. LM gratefully acknowledges support from the Max-Planck-Institute for Mathematics in Bonn.
	LW is supported by the RTG 1670 ``Mathematics inspired by String theory and Quantum
	Field Theory'' of the DFG.

\section{Equivariant field theories and their orbifolds}
In this section,
 we provide a brief and informal introduction to extended 
equivariant field theories and the topological orbifold construction. We try to give the
reader a general idea of the concepts involved. For more details,
we will give references to the literature.
Along the way, we will also 
present some slight extensions, in particular the technical Lemma~\ref{lemorbsum}.

\subsection{Equivariant field theories}\label{Sec: ETFT}
An \emph{$n$-dimensional topological  field theory}
(depending on the context also referred to as topological \emph{quantum} field theory)
as  defined by Atiyah~\cite{atiyah}
is a symmetric monoidal functor $Z\colon \Cob_n \longrightarrow \vs
$ from the symmetric monoidal category $\Cob_n$ of $n$-dimensional bordisms 
to the symmetric monoidal category $\vs$ of vector spaces (which in this note will always be over the complex numbers). 
Objects in $\Cob_n$ are $n-1$-dimensional closed oriented manifolds $\Sigma$ (we will work with \emph{oriented} manifolds throughout). A morphism $M:\Sigma_0 \to \Sigma_1$ is an equivalence class of compact 
oriented bordisms from $\Sigma_0$ to $\Sigma_1$, i.e.\ an equivalence class of compact oriented
$n$-dimensional manifolds $M$ with the choice of an orientation-preserving diffeomorphism 
$-\Sigma_0 \sqcup \Sigma_1 \cong \partial M$, where we denote by $-\Sigma_0$ the manifold $\Sigma_0$ with the reversed orientation.
Composition is given by gluing of
manifolds along boundaries. 

Over the last decades,
 this definition has been   generalized into various
directions. Two of those generalizations will be particularly relevant for us in the sequel:

\subparagraph{Equivariant topological field theories.} All manifolds appearing in the definition of $\Cob_n$ can be compatibly equipped with a map to the classifying space $BG$ of some (finite) group $G$ \cite{turaevhqft,htv,hrt}. This leads to a flavor of topological field theory that we will refer to as \emph{$G$-equivariant topological field theory}.
	
	\subparagraph{Extended topological field theories.} 
	By taking into account not only manifolds of dimension $n$ and $n-1$, but also $n-2$, one can define a symmetric monoidal \emph{bicategory} of cobordisms; we refer to \cite{schommerpries} for its definition and also for the notion of a symmetric monoidal bicategory in general. By slight abuse of notation, we will denote the bicategory of cobordisms by the same symbol $\Cob_n$ as its 1-categorical counterpart. For the rest of this note, $\Cob_n$ will always denote the $n$-dimensional cobordism \emph{bicategory}.
	
	 The symmetric monoidal category of vector spaces will also be replaced by an appropriate categorification. For our purposes, this will be the symmetric monoidal bicategory  $\Tvs$ of \emph{Kapranov-Voevodsky 2-vector spaces} \cite{KV,mortonvec}. Such a 2-vector space is a complex linear semisimple additive category with finite-dimensional morphism spaces and a finite number of isomorphism classes
	of simple objects. The 1-morphisms in $\Tvs$
	are linear functors, the 2-morphism are natural transformation. The monoidal 
	product is the Deligne product. Its monoidal unit is the category $\fvs$ of finite-dimensional complex vector spaces.
Now one defines an \emph{$n$-dimensional extended topological field theory} as a symmetric monoidal functor $\Cob_n \to \Tvs$
(a (symmetric monoidal) functor between (symmetric monoidal) bicategories is here to be understood as a (symmetric monoidal) 2-functor in the weak sense, in particular with plenty of coherence data).

In fact, one can replace $\Cob_n$ by an even higher category that takes into account even lower-dimensional manifolds  down to the point. This leads to \emph{fully extended topological field theories} \cite{lurietft}. In this sense, the extended topological field theories in this note will only be once-extended.  
One reason for our interest in topological field theories that are once-extended (but not non-extended or further extended) is the deep connection 
to representation theory (and particularly modular categories) that this 
 type of field theory exhibits in dimension three \cite{rt2,BDSPV153D}.

In the sequel, we will be interested in topological field theories which are both extended and equivariant as defined and investigated in \cite{extofk}. To this end, one needs to define a symmetric monoidal bicategory $G\text{-}\Cob_n$ of $G$-cobordisms for a group $G$.
We refer to \cite[Definition~2.1]{extofk} for the precise definition. Roughly,
$G\text{-}\Cob_n$ is defined as follows:
\begin{itemize}
\item An object is a closed oriented $n-2$-dimensional manifold $S$ together
with a continuous map $\xi : S\longrightarrow BG$.
\item A 1-morphism $(S_0,\xi_0)\longrightarrow (S_1,\xi_1)$ is a cobordism
$\Sigma \colon S_0 \longrightarrow S_1$ between $S_0$ and $S_1$ equipped with a continuous map
$\varphi\colon \Sigma \longrightarrow BG$ such that the diagram
\begin{equation}
\begin{tikzcd}
 & BG & \\ \\ 
S_0 \ar[r] \ar[ruu,"\xi_0"] & \Sigma \ar[uu,"\varphi"] & \ar[l] \ar[luu,"\xi_1", swap] S_1 
\end{tikzcd}
\end{equation}  
commutes. Composition is defined by gluing manifolds and continuous maps
along common boundary components.

\item A 2-morphism  $(\Sigma_a,\varphi_a)\longrightarrow (\Sigma_b,\varphi_b)$ 
between two 1-morphisms $(\Sigma_a,\varphi_a)$ and $(\Sigma_b,\varphi_b)$ from $(S_0,\xi_0)$ to $(S_1,\xi_1)$ is an equivalence class of
 a 2-cobordism
 $M\colon \Sigma_a 
\longrightarrow \Sigma_b$, i.e.\ a certain type of manifold with corners and an 
identification of $-\Sigma_a \sqcup \Sigma_b$ with part of its boundary, together
with a continuous map $\psi \colon M \longrightarrow BG$ compatible with 
$\varphi_1$ and $\varphi_2$. Horizontal and vertical composition are defined by the gluing of manifolds.
\end{itemize}  
The disjoint union of manifolds induces a symmetric monoidal structure on 
$G\text{-}\Cob_n$.
\begin{definition}
An \emph{$n$-dimensional extended $G$-equivariant field theory} is a symmetric monoidal functor
$Z:G\text{-}\Cob_n \longrightarrow \Tvs$ which has the \emph{homotopy
invariance property}, i.e.\ it satisfies $Z(M,\psi)=Z(M,\psi')$ for two 2-morphisms $(M,\psi)$ and $(M,\psi')$ between fixed 1-morphisms if $\psi$ is homotopic to $\psi'$ relative boundary.
We denote by $\gtft_n$ the bicategory of $n$-dimensional extended $G$-equivariant topological field theories. Its morphisms are symmetric monoidal transformations.
\end{definition}  

In fact, $\gtft_n$ is a 2-groupoid.

\begin{remark}\label{remdirsum}
The direct sum $Z_1\oplus Z_2$
	of equivariant
	field theories $Z_1,Z_2\colon G\text{-}\Cob_n \longrightarrow \Tvs$
	is again a
	$G$-equivariant topological field theory
	that can be defined using the following prescription from  \cite{durhuus} and \cite[I.1.5.3]{turaevhqft}: On  objects $(S,\xi) \in G\text{-}\Cob_n$ such that $S$ is non-empty and connected, we define  $(Z_1\oplus Z_2) (S,\xi) \coloneqq Z_1(S,\xi)\oplus Z_2(S,\xi) $. For the empty set, we set
	$Z_1\oplus Z_2 (\emptyset)\coloneqq \fvs$
	(this is needed for $Z_1\oplus Z_2$ to be symmetric monoidal). On non-connected manifolds, $Z_1\oplus Z_2$ is defined
	as the tensor product of its values on the connected components. Similarly,
	we define the value of $Z_1\oplus Z_2$ on connected morphisms as the direct
	sum of the independent theories (interpreted in an appropriate sense) 
and 
	extend this definition by tensor products to non-connected morphisms.
	\end{remark}

\begin{remark}\label{remhqft}
	In the above definitions, we can replace the space $BG$ with \emph{any} topological space $T$, see also \cite{extofk}, which in this context will then be referred to as the \emph{target space}. 
	We then obtain a symmetric monoidal bicategory $T\text{-}\Cob_n$ of cobordisms equipped with maps to $T$. A symmetric monoidal functor $T\text{-}\Cob_n \to \Tvs$ having the homotopy invariance property will be called an \emph{extended homotopy quantum field theory with target space $T$} based on the terminology used in \cite{turaevhqft}. These theories form again a 2-groupoid $T\text{-}\tft_n$. If we are given a disjoint union $S\sqcup T$ of spaces $S$ and $T$, the inclusions $S \to S\sqcup T$ and $T \to S\sqcup T$ induce an equivalence
	\begin{align}
	(S\sqcup T)\text{-}\tft_n \xrightarrow{\ \simeq  \ } S\text{-}\tft_n \times T\text{-}\tft_n \ . \label{eqnequivdisj}
	\end{align}
	For the proof, we observe $(S\sqcup T)^X \cong S^X\sqcup T^X$ for any \emph{connected} space $X$, where $Y^X$ denotes the space of maps $X\to Y$ equipped with the compact-open topology.
	Using this, we observe that
	 $(S\sqcup T)\text{-}\Cob_n$ is the coproduct of $S\text{-}\Cob_n$ and $T\text{-}\Cob_n$ as symmetric monoidal bicategories, and from this we deduce that \eqref{eqnequivdisj} is an equivalence.
	\end{remark}

\subsection{The topological orbifold construction}\label{Sec:Orbi}
The main result of~\cite{extofk} is a \emph{topological orbifold construction} for extended 
equivariant field theories: 
\begin{theorem}
For any finite group $G$,
there is a canonical functor
\begin{align}\label{eqnorbifoldconstruction}
- / G \ : \ \gtft_n
\quad 	\xrightarrow{ \ \ -/G \  \ } \quad 
\tft_n \ 
, \quad Z \longmapsto Z/G \ . 
\end{align}\normalsize
This functor is referred to as the {(topological) orbifold construction}. The image $Z/G$ 
of a $G$-equivariant topological field theory $Z$ under this functor is referred to as the orbifold theory of $Z$.
\end{theorem} 

As explained in \cite{extofk},
the orbifold construction provides a geometric underpinning and 
  generalization of known algebraic orbifoldization procedures such as the orbifoldization of crossed Frobenius algebras \cite{kaufmannorb} and categorical orbifoldization \cite{mueger,kirrilovg04,centerofgradedfusioncategories} including permutation orbifolds \cite{bantay98}.

We will give a rough description of the orbifold theory $Z/G$ of 
an extended $n$-dimensional $G$-equivariant topological field theory $Z$ with values in complex 2-vector spaces
and refer to \cite{extofk} for the details.
The orbifold construction is the composition of two constructions which are rather independent:

\begin{enumerate}
	
	\item

 In the first step, we assign to $Z$ an ordinary field theory
$\widehat{Z}\colon \Cob_n \longrightarrow \TVectBun$ with values in the symmetric monoidal bicategory
$\TVectBun$ 
of 2-vector bundles over essentially finite groupoids 
and their spans
as defined in~\cite{swpar}.
An object in $\TVectBun$ is an essentially finite groupoid $\Gamma$ (a groupoid with finitely many isomorphism classes and finite automorphism groups) together 
with a functor $\varrho \colon \Gamma \longrightarrow \Tvs$; we refer to this as a \emph{2-vector bundle}. A morphism 
from $(\Gamma_0,\varrho_0)$ to $(\Gamma_1,\varrho_1)$ consists of a span
of essentially finite groupoids 
\begin{equation}
\begin{tikzcd} 
 & \Gamma \ar[ld,"r_0", swap] \ar[rd, "r_1"] & \\
\Gamma_0 & & \Gamma_1
\end{tikzcd}
\end{equation} together with a natural transformation $r_0^*\varrho_0 \to 
r_1^*\varrho_1$, where $r_j^*\varrho_j$ for $j=0,1$ is the pullback of $\varrho_j$ along $r_j$
(in the applications that follow, the groupoids will be groupoids of bundles and the legs of the spans will arise from restriction of bundles to the boundaries of a cobordism). Composition is given by forming homotopy pullbacks. 
The 2-morphisms are equivalence classes of spans of spans of groupoids equipped
with modifications encoding the compatibility of the additional structure on the
spans. 

The field theory $\widehat{Z}$ assigns to an $n-2$-dimensional compact oriented manifold $S$ the
fundamental groupoid $\Pi(S,BG):=\Pi (BG^S)$ of the space $BG^S$ of maps $S\to BG$ equipped with the compact-open topology.
 An object in this groupoid is a continuous map
$S\longrightarrow BG$ and a morphism is a homotopy class of homotopies 
between maps into $BG$. A homotopy $h$ from $\xi_0\colon S \longrightarrow BG$
to $\xi_1 \colon S \longrightarrow BG$ can be seen as a 1-morphism $(S\times [0,1],h)\colon (S,\xi_0)\longrightarrow (S, \xi_1)$
in $G\text{-}\Cob_n$. Evaluating $Z$ on these 1-morphisms yields a 
functor $\Pi(S,BG)\longrightarrow \Tvs$, i.e.\ a 2-vector bundle that we define to be $\widehat{Z}(S)$.
The topological field theory $\widehat{Z}$ assigns to a 1-morphism $\Sigma\colon S_0 \longrightarrow S_1$ the span of groupoids
 \begin{equation}
\begin{tikzcd} 
 & \Pi(\Sigma, BG) \ar[ld,"r_0", swap] \ar[rd, "r_1"] & \\
\Pi(S_0,BG) & & \Pi(S_1,BG) \ , 
\end{tikzcd}
\end{equation} where $r_0$ and $r_1$ are the obvious restriction functors.
Evaluating $Z$ on   $\Pi(\Sigma, BG)$ induces the natural 
transformation which is part of the 1-morphism $\widehat{Z}(\Sigma)$.
We refer to the functor \begin{align} \widehat{-}\colon G\text{-}\catf{TFT}\longrightarrow 
[\Cob_n, \TVectBun]^\otimes\end{align} as the \emph{change to equivariant coefficients}. 
Here $[\Cob_n, \TVectBun]^\otimes$ is the 2-groupoid of symmetric monoidal functors $\Cob_n\to\TVectBun$ (i.e.\ $\TVectBun$-valued topological field theories).
For more details
and the definition on 2-morphisms, we refer to~\cite{extofk}.

\item The second ingredient of the orbifold construction is the bicategorical \emph{parallel section functor} $\Par \colon \TVectBun \longrightarrow \Tvs$ constructed in~\cite{swpar}, see \cite{trova} for a 1-categorical analogue. The parallel section functor sends a 2-vector bundle $\varrho \colon \Gamma\longrightarrow \Tvs$ to the homotopy limit of $\varrho$, i.e.\ the 2-vector space of homotopy fixed points. The value of $\Par$ on 1-morphisms and 2-morphisms
is constructed by a pull-push procedure for 2-vector bundles over  spans of bundle groupoids and the integral with respect to groupoid cardinality. 

\end{enumerate}
The orbifold theory is now defined by composing the change to equivariant 
coefficients with the parallel section functor, i.e.\ \eqref{eqnorbifoldconstruction} 
 factorizes as
\begin{align}
G\text{-}\catf{TFT}   \xrightarrow{\  \widehat{-}    \ }
[\Cob_n, \TVectBun]^\otimes \xrightarrow{\  \Par \circ -    \ } \tft_n \ . \label{eqndeforb}
\end{align}
It is clear that the orbifold construction with respect to a finite group $G$ (and all its ingredients) can be generalized to an orbifold construction 
\begin{align}\label{eqnorbifoldconstructionGamma}
\Gamma\text{-}\tft_n
\quad 	\xrightarrow{ \ \ -/\Gamma \  \ } \quad 
\tft_n \ 
, \quad Z \longmapsto Z/\Gamma  
\end{align}
 with respect to an essentially finite groupoid $\Gamma$.  Here we denote by $\Gamma\text{-}\tft_n$ the 2-groupoid of extended homotopy quantum field theories whose target space is the classifying space of the groupoid $\Gamma$ (Remark~\ref{remhqft}).
 The following Lemma says that the orbifold construction is compatible with the disjoint union of essentially finite groupoids:

 \begin{lemma}\label{lemorbsum}
 	For essentially finite groupoids $\Gamma$ and $\Omega$, the square
 		\begin{equation} 
 	\begin{tikzcd}
 	(\Gamma \sqcup \Omega)\text{-}\tft_n \ar[swap]{dd}{-/(  \Gamma\sqcup \Omega   )} \ar[]{rr}{\simeq } & &    \Gamma\text{-}\tft_n \times  \Omega\text{-}\tft_n  \ar{dd}{-/\Gamma\times - /\Omega}   \\
 	& & \\
 	\tft_n   & & \tft_n \times \tft_n \   \ar{ll}{\oplus } 
 	\end{tikzcd}
 	\end{equation} 
 	commutes, where the upper horizontal equivalence is from \eqref{eqnequivdisj}, the vertical arrows are the topological orbi\-foldizations with respect to $\Gamma$, $\Omega$ and $\Gamma\sqcup \Omega$, and the lower horizontal arrow is the direct sum operation from Remark~\ref{remdirsum}. 
 	\end{lemma}
 
 By the definition of the orbifold construction as the composition \eqref{eqndeforb}, the proof amounts to a compatibility of the direct sum with the change to equivariant coefficients and the parallel section functor which can be verified by a direct computation.

As a consequence of Lemma~\ref{lemorbsum}, the orbifold with respect to a possibly non-connected essentially finite groupoid can be expressed as a sum over orbifolds with respect to its connected components and hence as a sum over `ordinary' orbifold theories with respect to a finite group.

\spaceplease

\section{Dimensional reduction of equivariant topological field theories}\label{Sec: Dim}
Dimensional reduction is the process of producing $n$-dimensional field 
theories from $n+r$-dimensional ones by evaluating them on manifolds of 
the form $M\times T$,  where $T$ is a fixed closed (oriented) manifold of dimension $r$.

In the mathematical framework of topological field theories this can be realized
by pulling back a topological field theory $E\colon \Cob_{n+r}\longrightarrow \Tvs$ along the symmetric monoidal functor $-\times T \colon \Cob_{n}\longrightarrow \Cob_{n+r}$ for an $r$-dimensional closed oriented manifold $T$.

This construction does not directly generalize to equivariant field
theories $Z\colon G\text{-}\Cob_{n+r} \longrightarrow \Tvs$ since
there are multiple ways to construct 
a continuous map $M\times T \longrightarrow BG$ even after a map $M\longrightarrow BG$ is fixed. One possible way to 
construct such a map is to use the projection onto $M$. These maps would consequently be 
constant in the  direction of $T$. However, we will instead consider \emph{all} possible
ways to equip $M\times T$ with a $G$-bundle.

In order to make this precise, first observe that 
there is a symmetric monoidal functor 
\begin{align} -\times T \colon BG^T\text{-}\Cob_{n}\longrightarrow G\text{-}\Cob_{n+r}   \label{timesfunctoreqn}\end{align}
sending an object in 
$BG^T\text{-}\Cob_{n}$, i.e.\ an $n-2$-dimensional closed oriented manifold $S$
 equipped with a continuous map $\xi \colon S \longrightarrow 
BG^T$ to the manifold $S\times T$ equipped with the adjoint map
\begin{align}
S\times T \longrightarrow BG \ , \ \ s\times t \longmapsto \xi(s)(t) \ . 
\end{align} 
Note that this just uses the standard adjunction $-\times T \dashv -^T$.
The definition on 1-morphisms and 2-morphisms is accomplished completely analogously.

\begin{definition}
Let $Z\colon G\text{-}\Cob_{n+r} \longrightarrow \Tvs$ be an $n+r $-dimensional extended equivariant field theory and $T$ a closed oriented $r$-dimensional 
manifold. The \emph{$T$-reduction $\Red_T Z \colon BG^T\text{-}\Cob_{n} \longrightarrow \Tvs$} of $Z$ is 
defined as the precomposition
\begin{align}
\Red_T Z : BG^T\text{-}\Cob_{n} \xrightarrow{\  -\times T    \ }       G\text{-}\Cob_{n+r} \xrightarrow{\ Z \ } \Tvs 
\end{align}
 of $Z$ with the functor 
$-\times T $ from \eqref{timesfunctoreqn}. 
The dimensional reduction on $T$ takes the form of a functor
\begin{align} \Red_T : 	G\text{-}\catf{TFT}_{n+r} \to BG^T\text{-}\catf{TFT}_{n} \ . \end{align}
\end{definition}

Our main result is the compatibility of the orbifold construction
with dimensional reduction:

\begin{theorem}\label{Thm:main}
	Let $T$ be a closed oriented $r$-dimensional manifold and $G$ a finite group.
	Then the diagram
		\begin{equation} 
	\begin{tikzcd}[row sep = 1.1cm, column sep=1.1cm]
	G\text{-}\catf{TFT}_{n+r} \ar[swap]{d}{-/G} \ar[]{r}{\Red_T }  &   BG^T\text{-}\catf{TFT}_{n}  \ar{d}{-/BG^T}  
\\
\catf{TFT}_{n+r} \ar{r}{\Red_T }    & \catf{TFT}_{n} \   
	\end{tikzcd}
	\end{equation}
	commutes up to natural isomorphism.
\end{theorem}

Note here that the orbifoldization with respect to $BG^T$ (see \eqref{eqnorbifoldconstructionGamma}) makes sense because the space $BG^T$
can be seen as the classifying space of an essentially finite groupoid, see \cite[Lemma~2.8]{extofk}.

\begin{proof}By construction of the orbifoldization procedure, it suffices to prove that the two squares
		\begin{equation} 
	\begin{tikzcd}[row sep = 1.1cm, column sep=1.4cm]
	G\text{-}\catf{TFT}_{n+r} \ar[swap]{d}{\widehat{-}} \ar[]{r}{\Red_T }  &    BG^T\text{-}\catf{TFT}_{n}  \ar{d}{\widehat{-}}   \\
      \text{$[\Cob_{n+r}, \TVectBun  ]_\otimes$}      \ar[swap]{d}{\Par \circ - }   \ar[]{r}{\Red_T }       & \text{$[\Cob_n, \TVectBun  ]_\otimes$}    \ar{d}{\Par \circ - }  \\
	\catf{TFT}_{n+r} \ar{r}{\Red_T }   & \catf{TFT}_{n} \   
	\end{tikzcd}
	\end{equation}
	(in which the vertical compositions are (by definition) the orbifoldizations with respect to $G$ and $BG^T$, respectively)
	can be filled with a natural isomorphism. For the lower square, this is obvious because the precomposition and postcomposition with functors commute.
	Hence, it remains to prove that the upper square commutes up to natural isomorphism.

	In order to construct the component of the natural isomorphism at an object $Z\colon G\text{-}\Cob_{n+r}\longrightarrow \Tvs$ in
	$	G\text{-}\catf{TFT}_{n+r}$ recall that to an $n-2$-dimensional closed oriented manifold $S$ the $\TVectBun$-valued topological field theory $\widehat{\Red_T Z}$ assigns the 2-vector bundle $\Pi(S,BG^T) \xrightarrow{\  \Red_T Z(S,-)   \ } \TVectBun$. But as follows from $\Pi(S,BG^T)\cong \Pi(S\times T,BG)$ and  the definition of $\Red_T$, this 2-vector bundle is canonically isomorphic (in $\TVectBun$) 
	to the 2-vector bundle $\Pi(S\times T,BG)\xrightarrow{\    Z(S\times T,-)  \ } \TVectBun$. These isomorphisms extend analogously to 1-morphisms and 2-morphisms and yield a canonical isomorphism $\widehat{\Red_T Z}\cong \Red_T \widehat{Z}$. These isomorphisms assemble the natural isomorphism that we need to fill the upper square.
	\end{proof}

For any $n+r$-dimensional extended $G$-equivariant topological field theory $Z\colon G\text{-}\Cob_{n+r}\longrightarrow \Tvs$,
we obtain by the above result a canonical isomorphism \begin{align} \Red_T \frac{Z}{G} \cong  \frac{\Red_T Z}{ BG^T  } \label{eqncomputerhs} \ . \end{align} 
 By \cite[Lemma~2.8]{extofk} we have $BG^T \simeq B \PBun_G(T)\simeq \bigsqcup_{[P] \in \pi_0  (\PBun_G(T))  } B \Aut(P)$, where $\PBun_G(T)$ is the groupoid of $G$-bundles over $T$. Therefore, the $BG^T$-equivariant theory $\Red_T Z$ decomposes into a finite family $\left\{    (  \Red_T Z )^{[P]}\right\}_        {[P] \in \pi_0  (\PBun_G(T))}$ of $\Aut(P)$-equivariant field theories under the equivalence \eqref{eqnequivdisj}. 
 We can use Lemma~\ref{lemorbsum} to express the right hand side of \eqref{eqncomputerhs} in terms of the orbifold theories of the members of this family and hence arrive at the following decomposition:

\begin{theorem}\label{thmdimred2}
	Let $T$ be a closed oriented $r$-dimensional manifold, $G$ a finite group and $Z\colon G\text{-}\Cob_{n+r}\longrightarrow \Tvs$ an $n+r$-dimensional extended $G$-equivariant topological field theory. Then we have a canonical isomorphism
	\begin{align}
	\Red_T \frac{Z}{G} \cong \bigoplus_{[P] \in \pi_0(\PBun_G(T))   }    \frac{(\Red_T Z)^P}{\Aut(P)}     \label{eqnmainthm} 
	\end{align}
	of $n$-dimensional extended topological field theories.
	\end{theorem}

For $T=\mathbb{S}^1$ (one of the most common choices for dimensional reductions), the bundle groupoid $\PBun_G(\mathbb{S}^1)$ is equivalent to the loop groupoid $G//G$ of $G$, i.e.\ the action groupoid for the action of $G$ on itself by conjugation. As a consequence, we find
	\begin{align}\label{eqnsphere}
\Red_{\mathbb{S}^1} \frac{Z}{G} \cong \bigoplus_{ [g] \in G/G   }    \frac{(\Red_   {\mathbb{S}^1}  Z)^g}{\cen (g) } \ ,   
\end{align}
where $G/G$ is the set of conjugacy classes of $G$ and $\cen(g)\subset G$ is the centralizer group of an element $g\in G$. 
Since the centralizer of the neutral element is $G$, the dimensional 
reduction of the orbifold theory $Z/G$ 
contains the summand  $(\Red_{\mathbb{S}^1} Z   )^e/G$  corresponding
to the na\" ive reduction that just allows $G$-bundles which are trivial in  $\mathbb{S}^1$-direction.

\begin{remark}\label{remcatint}
	Let $\Gamma$ be an essentially finite groupoid and $f: \Gamma \to \mathbb{C}$ and invariant function on $\Gamma$, i.e.\ a complex-valued function on the object set such that $f(x)\cong f(y)$ whenever $x\cong y$ for $x,y \in \Gamma$. Then we can define the \emph{integral of $f$ with respect to groupoid cardinality} by
	\begin{align} \int_{ \Gamma} f(x)\,\text{d} x := \sum_{[x]\in \pi_0(\Gamma)} \frac{f(x)}{|\! \Aut(x)|} \ , \label{eqnintgrpdcard}      \end{align} see \cite{baezgroup} for more background on this primitive, but surprisingly powerful integration theory. 
	
	It is hard to overlook the formal resemblance between this integral and the right hand side of \eqref{eqnmainthm}. In fact, we can interpret the right hand side of \eqref{eqnmainthm} as a \emph{categorified integral} over the essentially finite groupoid $\PBun_G(T)$, where the `function' to be integrated is $P \mapsto (\Red_T Z)^P$. This function takes values in field theories. 
	Compared to \eqref{eqnintgrpdcard}, the sum is replaced by a direct sum and the division by the sizes of automorphism groups is replaced by topological orbifoldization. 
	\end{remark}

\section{Application to Dijkgraaf-Witten theories}\label{Sec: DW}
For any group $G$, a $d$-cocycle $\omega\in Z^d(BG;\U(1))$ gives rise to a $d$-dimensional $G$-equivariant topological field theory.
The non-extended version of this field theory is constructed in \cite[I.2.1]{turaevhqft}. 
The extended version that we will denote by
\begin{align}
E_\omega : G\text{-}\Cob_d \to \Tvs \label{eqnprimitive}
\end{align} 
is constructed in \cite{MuellerWoike1}. 
 Its value on a closed oriented $d$-dimensional manifold $M$ 
equipped with a map $\psi \colon M \longrightarrow BG$ 
is the number $\langle \psi^* \omega,\sigma_M \rangle \in \U(1)$, where $\sigma_M$ is a representative for the fundamental class of $M$ and  $\langle -, - \rangle$ is the  pairing between cochains
and chains. This number is often suggestively written as an integral
$\int_M \psi^* \omega \in \U(1)$. 
Both in \cite{turaevhqft} and \cite{MuellerWoike1}, these field theories can be constructed from cocycles on arbitrary spaces.

In order to describe the dimensional reduction of the field theories \eqref{eqnprimitive},
 we briefly recall
the concept of transgression. Let $T$ be a $r$-dimensional closed 
oriented manifold and $\sigma_T\in Z_r(T)$ a representative of its 
fundamental class. 
The transgression of a $d$-cocycle $\omega$ on $BG$ to $BG^T$ is
the $d-r$ cocycle $\tau \omega$ on $BG^T$ given by
\begin{align}
(\tau \omega) (s)= \ev^*\omega (s\times \sigma_T) \ \ ,
\end{align}  
where $s$ is a $d-r$-chain on $BG^T$
and $\ev \colon BG^T\times T \longrightarrow BG$ is the evaluation map.

\begin{proposition}\label{proptransgression}
Let $E_\omega$ be the extended
$G$-equivariant topological field theory from \eqref{eqnprimitive} associated to a cocycle $\omega\in Z^{n+r}(BG;\U(1))$ 
and $T$ a closed oriented $r$-dimensional manifold. Then the $T$-reduction $\Red_T E_\omega\colon BG^T\text{-}\Cob_n\longrightarrow \Tvs$ 
is naturally isomorphic to $E_{\tau \omega}$.
\end{proposition}

\begin{proof}
We only present the proof on the level of closed oriented $n$-dimensional manifolds.
Similar arguments can be applied to the full theories. Let $M$ be a closed $n$-dimensional manifold 
equipped with a map $\psi \colon M \longrightarrow BG^T$.
Then the value of $\Red_T E_\omega$ on $(M,\psi)$ is 
\begin{align}
\int_{M\times T} (\ev \circ (\psi\times \id_T))^* \omega = \langle  \ev^* \omega, \psi_* \sigma_M \times \sigma_T  \rangle = \langle \tau \omega, \psi_*\sigma_M \rangle = \int_M \psi^* (\tau \omega) \ \ , 
\end{align}  
where $\sigma_M$ and $\sigma_T$ is a representative for the fundamental class of
$M$ and $T$, respectively.
\end{proof}

The singular cohomology of an aspherical space can be identified with
the groupoid cohomology of its fundamental groupoid. We refer to
\cite{willerton} for a discussion and concrete formulae for the transgression
to loop groupoids. 

Dijkgraaf-Witten theories 
\cite{dw,freedquinn,morton1}
are topological field theories which describe gauge
theories with finite structure group $G$. 
The twisted $n$-dimensional  Dijkgraaf-Witten theory
that we denote by
\begin{align}
\DW _G^\omega:\Cob_n \to \Tvs 
\end{align}
 depends on the choice of a topological action functional 
which can be described by a $\U(1)$-valued $n$-cocycle $\omega$ on $BG$. 
The relation between Dijkgraaf-Witten theories and the content of this note is the fact that they can be written as an orbifold theory,
\begin{align}
\DW _G^\omega\cong \frac{E_\omega}{G} \ , \label{eqndworb}
\end{align}
see
\cite[Remark~4.5~(1)]{swofk} for the statement on the level of non-extended field theories and
\cite[Section~5]{MuellerWoike1} for the extended case. 
From Theorem~\ref{thmdimred2}, Proposition~\ref{proptransgression} and \eqref{eqndworb} we directly conclude:

\begin{theorem}
For any	 closed oriented $r$-dimensional manifold $T$, any finite group $G$ and any cocycle $\omega \in Z^{n+r}(BG;\U(1))$, the $T$-reduction of the $\omega$-twisted Dijkgraaf-Witten theory $\DW_G^\omega$ is given by
\begin{align}
\Red_T \DW_G^\omega  \cong \bigoplus_{[P] \in \pi_0(\PBun_G(T))   }  \DW_{\Aut(P)}^{\tau \omega|_{   \Aut (P)  }}   \ .
\end{align}
\end{theorem}

In the special case
$T=\mathbb{S}^1$ that was treated in \eqref{eqnsphere},
we prove a recent conjecture of Qiu-Wang \cite{qw}:

\begin{corollary}\label{correddw}
The dimensional reduction of the $\omega$-twisted Dijkgraaf-Witten theory $\DW_G^\omega$ is given by
\begin{align}\label{Eq: DW S1}
\Red_{    \mathbb{S}^1       } \DW_G^\omega \cong \bigoplus_{[g] \in G/G   }  \DW_{\cen(g)}^{\tau \omega|_{   \cen(g)  }}  \ .
\end{align}
\end{corollary}

The precise conjecture in \cite{qw} concerns the untwisted case and a reduction from dimension $(4,3,2)$ to $(3,2,1)$.
The statement, however, holds independently of the dimension and also with cohomological twists.

\begin{remark}
We conclude with a few comments on the relation of our results to 
Conjecture~1.2 of \cite{KTZ} which says that the center of the $n$-category $n\catf{Vect}_G^\omega$
of $G$-graded $n$-vector spaces twisted by an $n+2$-cocycle $\omega\in Z^{n+2}(BG;\U(1))$ is given by
\begin{align}\label{Eq: Conjecture}
Z(n\catf{Vect}_G^\omega) \cong \bigoplus_{[g]\in G/G} n\catf{Rep}^{\tau \omega|_{\cen(g)}}(\cen(g)) \ ,              
\end{align} 
where $n\catf{Rep}^{\tau \omega|_{\cen(g)}}(\cen(g))$ is the category of $\tau \omega|_{\cen(g)}$-twisted $n$-representations of $\cen(g)$. 

The fact that the right hand sides of \eqref{Eq: DW S1} and \eqref{Eq: Conjecture} have a similar form
is not a coincidence. 
In fact, based on Corollary~\ref{correddw}, we can suggest a strategy for a topological proof of \eqref{Eq: Conjecture}:
 It is expected that the $\omega$-twisted Dijkgraaf-Witten theory for the group $G$ admits an extension to a fully extended field theory with values in some suitable target category. Furthermore, the algebraic object describing this fully
extended $n+2$-dimensional field theory (its value at the point) is expected to be $n\catf{Vect}_G^\omega$. One expects 
that there also exists a dual description of the field theory in terms of the 
higher category $(n+1)\catf{Rep}^\omega(G)$. The relation between the
two descriptions is by identifying $(n+1)\catf{Rep}^\omega(G)$ with
the higher category of modules over $n\catf{Vect}_G^\omega$.  
If we now assume that a result analogous to \eqref{Eq: DW S1} can be established for the dimensional reduction of the fully extended Dijkgraaf-Witten theory on $\mathbb{S}^1$, we can evaluate both sides on the point. 
On the left hand side, we would find the value of the fully extended Dijkgraaf-Witten on the circle, which by general principles in topological field theories would be the center $Z(n\catf{Vect}_G^\omega)$. The right hand side would be the right hand side of \eqref{Eq: Conjecture} by the mentioned description of Dijkgraaf-Witten theories in terms of twisted representations.
Making this argument or even any of its ingredients precise is far beyond the scope of this short note.    
\end{remark}

\end{document}